\documentclass{elsarticle}
\usepackage{amsthm,enumerate, graphicx, verbatim}
\usepackage[mathscr]{euscript}
\newtheorem{thm}{Theorem}[section]

\newtheorem{lem}[thm]{Lemma}

\theoremstyle{definition}

    \makeatletter
    \def\ps@pprintTitle{%
       \let\@oddhead\@empty
       \let\@evenhead\@empty
       \let\@oddfoot\@empty
       \let\@evenfoot\@oddfoot
    }
    \makeatother

\begin{document}
\title{Edges and Vertices in a Unique Signed Circle in a Signed Graph}
\author{Richard Behr}
\address{Department of Mathematics, Binghamton University, Binghamton, NY 13902}
\ead{behr@math.binghamton.edu}
 \begin{abstract} We examine the conditions under which a signed graph contains an edge or a vertex that is contained in a unique negative circle or a unique positive circle. For an edge in a unique signed circle, the positive and negative case require the same structure on the underlying graph, but the requirements on the signature are different. We characterize the structure of the underlying graph necessary to support such an edge in terms of bridges of a circle. We then use the results from the edge version of the problem to help solve the vertex version.
\end{abstract}
\maketitle
\section*{Introduction}
A \emph{signed graph} is a graph in which each edge is assigned either a positive or negative sign. The sign of a circle (a connected, $2$-regular subgraph)  in a signed graph is defined to be the product of the signs of its edges. In many cases, the most important feature of a signed graph is the sign of each of its circles.  A signed graph that contains no negative circle is said to be \emph{balanced}, while a signed graph that contains at least one negative circle is \emph{unbalanced}. The purpose of this paper is to determine when a signed graph contains an edge or a vertex that is contained in a unique negative circle or a unique positive circle.

Signed graphs were invented by Harary in 1953 in order to help study a question in social psychology \cite{harary}. In 1956, Harary observed that an edge of a signed graph lies in some negative circle if and only if the block (maximal $2$-connected subgraph) containing it is unbalanced \cite{harary2}. Similarly, an edge lies in some positive circle if and only if it is not a balancing edge in its block (see Lemma \ref{negativebalance}). Our problem is related to these facts, but the added uniqueness condition creates many additional restrictions on both the structure of the underlying graph and the signature.

\section{Definitions}
\subsection{Graphs}
A \emph{graph} $G = (V(G),E(G))$ consists of a finite vertex set $V(G)$ and finite edge set $E(G)$. Each edge has a pair of vertices as its \emph{endpoints}, and we write $e{:}uv$ for an edge with endpoints $u$ and $v$. A \emph{link} is an edge with two distinct endpoints, and a \emph{loop} has two equal endpoints. We write $K_n$ for the complete graph on $n$ vertices. 

Let $H$ be a subgraph of $G$. Then for $v \in V(G)$, the \emph{degree of} $v$ \emph{in} $H$, denoted $\deg_H(v)$, is the number of edges in $H$ that are incident with $v$ (a loop counts twice).

A \emph{circle} $C$ is a connected $2$-regular subgraph. An edge $e \in E(G){\setminus}E(C)$ connecting two different vertices of $C$ is a \emph{chord}.

A \emph{path} $P=v_0,e_0,v_1,e_1,...,e_{n-1},v_n$ is a sequence of adjacent vertices and connecting edges that never repeats an edge or a vertex. We call $v_0$ and $v_n$ the \emph{endpoints} of $P$, while the other vertices are \emph{interior vertices}. 

We \emph{subdivide} an edge by replacing it with a path that has at least one edge. A \emph{subdivision} of $G$ is a graph obtained by subdividing some of the edges of $G$.

Given a circle $C$ of $G$, a \emph{bridge} of $C$ is either a connected component $D$ of $G{\setminus}V(C)$ along with all edges joining $D$ to $C$, or a chord of $C$. The \emph{vertices of attachment} of a bridge $D$ are the vertices in $V(D) \cap V(C)$. A path contained in $D$ that has different vertices of attachment for its endpoints is a \emph{path through $D$}.

A \emph{cutpoint} of $G$ is a vertex $v$ with the property that there exist subgraphs $H_1$ and $H_2$ each with at least one edge, such that $G=H_1 \cup H_2$ and $H_1 \cap H_2 = \{v\}$. An \emph{isthmus} is an edge whose removal increases the number of connected components. A \emph{block} of $G$ is a maximal subgraph that contains no cutpoint. Each edge is contained in exactly one block.  

\subsection{Signed Graphs}
A \emph{signed graph} $\Sigma$ is a pair $(G, \sigma)$, where $G$ is a graph (called the \emph{underlying graph}), and $\sigma : V(G) \rightarrow \{+,-\}$ is the \emph{signature}.

The \emph{sign} of a circle $C$ in $\Sigma$ is defined to be the product of the signs of its edges. Thus, a signed circle can be either positive or negative. A signed graph is \emph{balanced} if all of its circles are positive, and \emph{unbalanced} if it contains at least one negative circle.

A \emph{theta graph} consists of three paths with the same endpoints and no other vertices in common. The most useful thing about theta graphs in our context is the \emph{theta property}: every signed theta graph has either $1$ or $3$ positive circles. If two circles $C_1$ and $C_2$ intersect in a path with at least one edge, then $C_1 \cup C_2$  is a theta graph with third circle $C_1 \Delta C_2$ (we use $\Delta$ for symmetric difference). By the theta property, if $C_1$ and $C_2$ have the same sign then $C_1 \Delta C_2$ is positive, and otherwise $C_1 \Delta C_2$ is negative.

A \emph{switching function} on $\Sigma=(G,\sigma)$ is a function $\zeta : V(G) \rightarrow \{+,-\}$.  We can use $\zeta$ to modify $\sigma$, obtaining a new signature given by $\sigma^{\zeta}(e):=\zeta(v) \sigma(e) \zeta(w)$, where $v,w$ are the endpoints of $e$. The switched signed graph is written $\Sigma^{\zeta}:=(G,\sigma^{\zeta})$.  If $\Sigma '$ is obtained from $\Sigma$ via switching, we say $\Sigma'$ and $\Sigma$ are \emph{switching equivalent}, written $\Sigma' \sim \Sigma$. Switching is useful for us because of the following fact.

\begin{lem}[{Zaslavsky \cite{zaslav1}}, Soza{\'n}ski \cite{soz}]\label{sw2}Let $\Sigma_1$ and $\Sigma_2$ be signed graphs on the same underlying graph. Then, $\Sigma_1 \sim \Sigma_2$ if and only if $\Sigma_1$ and $\Sigma_2$ have the same collection of positive circles. In particular, $\Sigma$ is balanced if and only if it switches to an all-positive signature.
\end{lem}

If $\Sigma$ can be switched so that it has a single negative edge $b$, we call $b$ a \emph{balancing edge}. The deletion of a balancing edge yields a balanced signed graph. Moreover, if $b$ is a balancing edge, 
the negative circles of $\Sigma$ are precisely those that contain $b$.

Assume an edge $e$ is contained in at least one circle. Then $e$ is contained in only positive circles if and only if the block containing it is balanced, and $e$ is contained in only negative circles if and only if it is a balancing edge in its block (Lemma \ref{negativebalance}). In other words, $e$ is contained in some negative circle if and only if the block containing $e$ is unbalanced (as discovered by Harary \cite{harary2}), and $e$ is contained in some positive circle if and only if $e$ is not balancing in the block containing it.
\section{Batteries}

Let us give a name to the main object of study in this paper. A \emph{battery} is an edge of $\Sigma$ that is contained in a unique negative or positive circle $C$. An edge $e$ may be a \emph{negative battery} or a \emph{positive battery}, depending on the sign of $C$. We write $[e,C,-]$ for a negative battery $e$ that is contained in the unique negative circle $C$, and similarly $[e,C,+]$ for a positive battery. The advantage of this notation is that it enables us to keep track of both $C$ and its sign.

We wish to determine when a given edge $e$ is a battery. This problem is uninteresting when $e$ is an isthmus of $\Sigma$ ($e$ is not contained in a circle), so throughout we assume $e$ is contained in at least one circle. Since the block $B$ containing $e$ is the union of all circles containing $e$, we will focus only on $B$. If $B$ consists of a single circle, it is trivially true that all edges of $B$ are batteries. Thus, for the rest of this section we assume that the underlying graph of $\Sigma$ is a block $B$ containing $e$ such that $B$ is neither an isthmus, nor a circle.

\subsection{Layering} In this section we introduce the structure of the underlying block $B$ that is necessary for $e$ to be a battery. 

First, we need some terminology to help describe the chords of a circle $C$. We say that chords $c_1$ and $c_2$ of $C$ \emph{cross} if $C \cup c_1 \cup c_2$ is a subdivision of $K_4$. Chords that do not cross are \emph{noncrossing}.  Let $\mathcal{E}$ be the set of all vertices of $C$ that are endpoints of some chord of $C$. The vertices of $\mathcal{E}$ partition $C$ into paths, called \emph{segments}. A segment of $C$ whose endpoints have a chord between them is called a \emph{handle}. We say that $C$ is \emph{$2$-handled} if any two of its chords are noncrossing and it contains exactly two handles. Thus, the property of being $2$-handled is a stronger version of having noncrossing chords (noncrossing chords may create many handles). If $C$ is $2$-handled, edges $e$ and $f$ contained in separate handles are said to be \emph{separated}. 

We will now define the appropriate structure on $B$, and then explain how to reduce this structure to a $2$-handled circle. Let $C$ be a circle of $B$, and suppose every bridge of $C$ has exactly two vertices of attachment. The graph obtained by replacing each of these bridges with a single chord between its vertices of attachment is denoted $C^*$. If $C^*$ is $2$-handled, we say that $B$ is \emph{$C$-layered}. See Figure \ref{f1} for an example of a $C$-layered graph. 

\begin{figure}[h!]
\centering
\includegraphics[scale=.20]{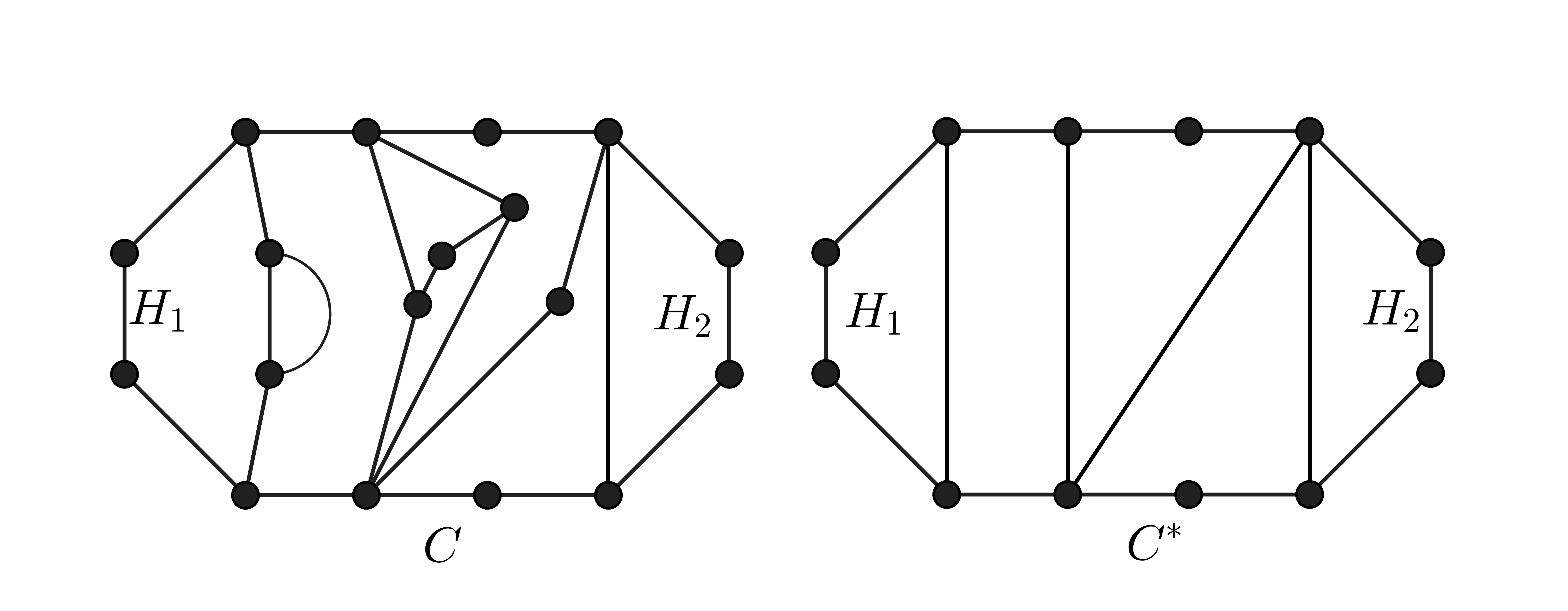}
\caption{On the left, we have a $C$-layered graph. On the right, we have replaced all bridges of $C$ with chords to obtain the $2$-handled circle $C^*$. In each picture, the handles are labeled $H_1$ and $H_2$, while the other $4$ segments of $C$ are unlabeled. \label{f1}}
\end{figure} 

Assuming that each bridge of $C$ has two vertices of attachment, it is extremely convenient to use $C^*$ to deduce things about $B$. For example, suppose $C^*$ has a theta graph containing chords $c_1$, $c_2$, and $c_3$, and suppose $c_1$, $c_2$, and $c_3$ correspond to the bridges $D_1$, $D_2$, and $D_3$. Then, we can deduce that $B$ contains at least one theta graph whose three constituent paths go through $D_1$,$D_2$, and $D_3$ (this occurs in Figure \ref{f1}---pick any three chords). The following lemma makes this example concrete.

\begin{lem}\label{subd} Let $H$ be a subgraph of $C^*$. Then, $B$ contains a subdivision of $H$. 
\begin{proof} To obtain the subdivision of $H$ in $B$, first take $H \cap C$. Then, for each chord in $H$, take a path through the corresponding bridge. 

\end{proof}
\end{lem}

We will frequently make use of Lemma \ref{subd} by finding configurations of circles in $C^*$ and then pulling them up to $B$.

\subsection{Negative and Positive Batteries}

We are now ready to describe when a certain fixed edge is a battery. As a reminder, we assume that $\Sigma=(B, \sigma)$ is a signed graph such that $B$ is a block that is not an isthmus, nor circle. Also, $e$ is an edge of $\Sigma$, and $C$ is a circle containing $e$. Our main theorem concerning edge batteries is as follows.

\begin{thm}\label{battery1} The edge $e$ is a battery if and only if $B$ is $C$-layered, $e$ is contained in a handle of $C$, and either:

\begin{enumerate}
\item The other handle of $C$ contains a balancing edge, in which case $[e,C,-]$ is a negative battery.
\item Every path through a bridge of $C$ makes a negative circle with either half of $C$, in which case $[e,C,+]$ is a positive battery.
\end{enumerate}
\end{thm}

First, we need a lemma. The author thanks Thomas Zaslavsky for suggesting this lemma, which helps simply our proof method.

\begin{lem}\label{tomk4} Every edge in a signed subdivision of $K_4$ lies in an even number of negative circles and an even number of positive circles.

\begin{proof}  Let $\Sigma$ be a signed subdivision of $K_4$, and let $e$ be an edge of $\Sigma$. Then, $e$ is contained in exactly four circles, denoted $C_1, C_2, C_3$, and $C_4$. The labeling is chosen so that $C_1 \Delta C_2 = C_3 \Delta C_4$. If $C_1 \Delta C_2$ is a positive circle, then $C_1$ and $C_2$ must have the same sign, as well as $C_3$ and $C_4$ (here we are using the theta property). If $C_1 \Delta C_2$ is a negative circle, then $C_1$ and $C_2$ have different signs, as do $C_3$ and $C_4$. In either case, there are an even number of both positive and negative circles containing $e$. 
\end{proof}
\end{lem}

\begin{proof}[Proof of Theorem \ref{battery1}] The reverse direction of the proof is easy. If $B$ is $C$-layered and $e$ is contained in a handle, then clearly either of the two signature properties listed imply that $e$ is a negative or positive battery respectively. 

It remains to show the forward direction. Suppose $e$ is a battery contained in the unique signed circle $C$. If there exists some bridge $D$ of $C$ that has three or more vertices of attachment, then $D \cup C$ contains a subdivision of $K_4$ (containing all of $C$). By Lemma \ref{tomk4}, $e$ is contained in at least two circles that have the same sign as $C$ in this subdivision---impossible by assumption that $e$ is a battery in $C$. Thus, each bridge of $C$ must have precisely $2$ vertices of attachment (they can't have $1$, since $G$ is a block).

Replace each bridge of $C$ with a single chord between its vertices of attachment to form $C^*$. We will find configurations of circles in $C^*$ and repeatedly use Lemma \ref{subd} without mention to pull them up to $B$. If $C$ has two chords $c_1$ and $c_2$ that are crossing, then $C \cup c_1 \cup c_2$ is a subdivision of $K_4$ containing $C$, once again impossible by Lemma \ref{tomk4}. So we assume that $C^*$ is a circle with noncrossing chords. Now we want to prove that $C$ has exactly two handles and that $e$ must be in one of them. 

First, suppose that $e$ is contained in a segment of $C$ that is not a handle. Thus, the segment of $C$ containing $e$ meets two different chords at its endpoints, $c_1$ and $c_2$. Consider $C \cup c_1 \cup c_2$. This graph has six circles, denoted $C, C_1, C_2, C_3, C_4$, and $C_5$. We choose the labeling so that $e \in C,C_3,C_4,C_5$ and $C_1 \Delta C_2 \Delta C_3 = C$ and $C_2 \Delta C_3 = C_4$, and $C_1 \Delta C_3 = C_5$. This situation is illustrated in Figure \ref{f2}. If $C$ is negative, then $C_3$, $C_4$, and $C_5$ must be positive (they contain $e$). Since $C_3 \Delta C_5 = C_1$ and $C_4 \Delta C_5 =C_2$, both $C_1$ and $C_2$ are positive. This implies that $C=C_1 \Delta C_2 \Delta C_3$ is positive, a contradiction. Similarly, if $C$ is positive, then $C_3$, $C_4$ and $C_5$ are negative. This implies that $C_1$ and $C_2$ are positive, but then $C_1 \Delta C_4 =C$ implies that $C$ is negative, a contradiction. Thus, this case is impossible---$e$ must be contained in a handle.

\begin{figure}[h!]
\centering
\includegraphics[scale=1]{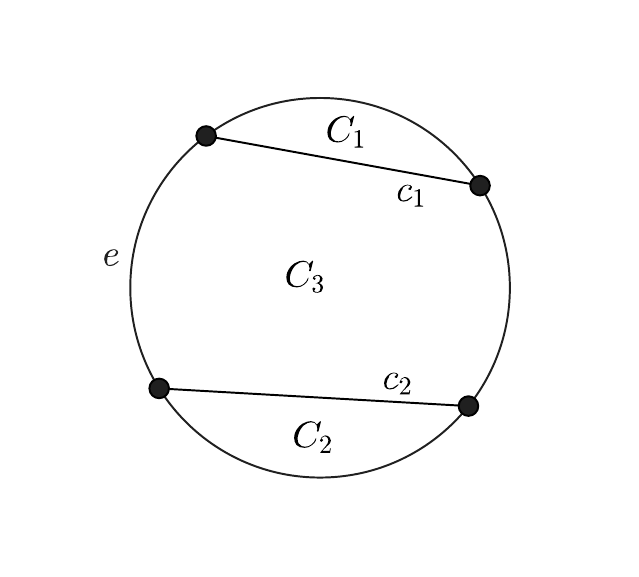}
\caption{The edge $e$ is not contained in a handle of $C$. The circles $C_2 \Delta C_3 = C_4$, and $C_1 \Delta C_3 = C_5$ are not labeled. \label{f2}}
\end{figure} 

Now, we have $C^*$ with non-crossing chords and we know $e$ is in a handle of $C$. All that remains is to show that there are exactly two handles. Suppose there are (at least) three handles, and let their corresponding chords be $c_1$, $c_2$, and $c_3$. Each handle along with its chord forms a circle, and we denote these $C_1$, $C_2$, and $C_3$, respectively. We choose labeling so that $e \in C_1$, and we write $C_4:= C \Delta C_1 \Delta C_2 \Delta C_3$. This situation is illustrated in Figure \ref{f3}. If $C$ is negative, then $C_1$ and $C_1 \Delta C_4$ and $C_1 \Delta C_4 \Delta C_3$ are positive, which means $C_3$ (and similarly $C_2$) are positive. However, $(C_1 \Delta C_4 \Delta C_3) \Delta C_2 = C$ which implies that $C$ is positive, a contradiction. Similarly, if $C$ is positive then $C_1 \Delta C_4 \Delta C_3$ is negative and $C_2$ and $C_3$ must again be positive. Therefore, the equality $(C_1 \Delta C_4 \Delta C_3) \Delta C_2 = C$ contradicts the fact that $C$ is positive. Consequently, there must be no more than two handles present. Note that under the assumption that $B$ is not a single circle, less than two handles are also impossible.

\begin{figure}[h!]
\centering
\includegraphics[scale=1]{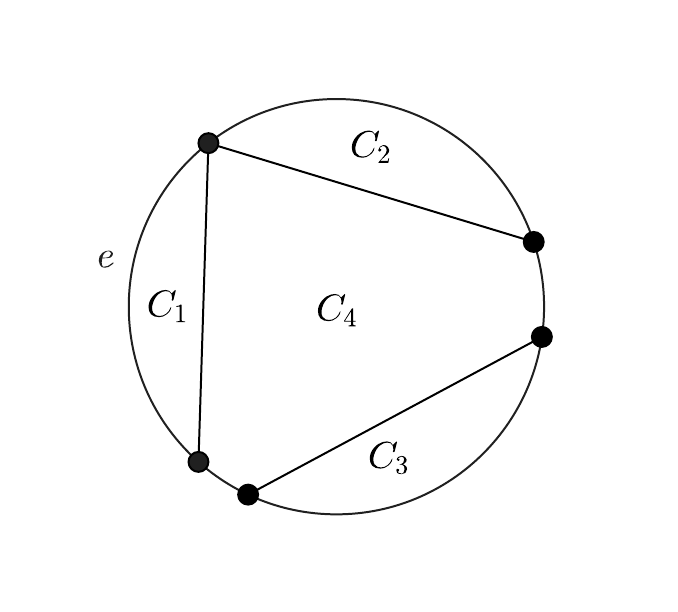}
\caption{The edge $e$ is contained in a handle, and there are two other handles. \label{f3}}
\end{figure}

Now that we have confirmed the structure of $B$, we will examine the signature of $\Sigma$. We break this into two cases. First, suppose $C$ is negative (so that $e$ is a negative battery). Delete the handle of $C$ not containing $e$ (denoted $H$), destroying the only negative circle containing $e$. The resulting $\Sigma {\setminus} H$ is still a block (clearly deleting a handle from a $C$-layered graph with at least one bridge does not create a cutpoint), and hence balanced (it contains $e$). Thus, we can switch $\Sigma$ so that the only negative edge is contained in $H$ (Lemma \ref{sw2}). This is the balancing edge described in the theorem.

If $C$ is positive, we observe that every path through a bridge of $C$ between its vertices of attachment makes a negative circle with either half of $C$. To see this, it is convenient to switch $\Sigma$ so that $C$ is all-positive. Then, any path through a bridge of $C$ between vertices of attachment must have a negative sign, otherwise this path along with $C$ forms an all-positive theta graph containing $e$ (and hence two positive circles containing $e$). 
\end{proof}

Consider a negative battery $[e,C,-]$ separated from a balancing edge $b$, and let $H$ be the handle containing $b$. Clearly, every edge in $H$ is a balancing edge. Moreover, no edge outside of $H$ is a balancing edge. This is because $G {\setminus} H$ is a block, and thus $e$ is contained in a circle besides $C$ with any given edge in $G{\setminus}H$. Thus we call $H$ the \emph{balancing handle}, since it contains all balancing edges of $\Sigma$. The presence of a single negative battery fixes $H$, and as a consequence, we are able to describe the location of all other negative batteries at the same time. 

\begin{thm}\label{allbat} Suppose $\Sigma$ has a negative battery $[e,C,-]$ separated from a balancing handle $H$. Let $f$ be an edge contained in $\Sigma{\setminus}C$ and let $D$ be the bridge of $C$ containing $f$. Then $f$ is a negative battery if and only if there is exactly one path through $D$ that contains $f$. 
\begin{proof} If $f$ is a negative battery there must only be one path through $D$ containing it, since each path through $D$ can be extended to a circle containing $H$. Conversely, if there is exactly one path through $D$ that contains $f$, this path extends uniquely to a circle containing $H$. 

\end{proof}

\end{thm}
\begin{figure}[h!]
\centering
\includegraphics[scale=.35]{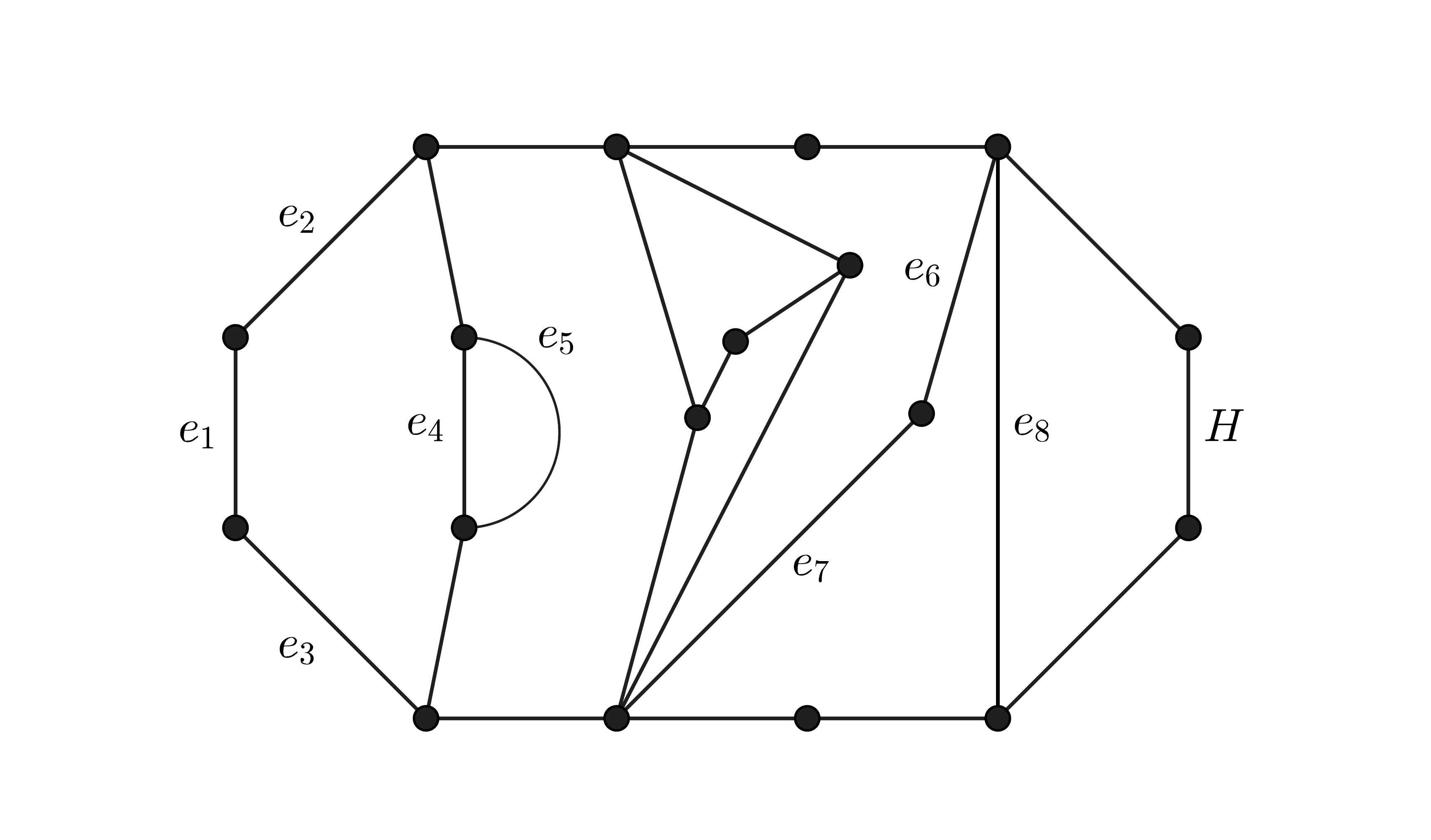}
\caption{Here, $H$ is the balancing handle. The batteries are labeled $e_1,...,e_8$. Each battery outside of $C$ is contained in a unique path through the bridge containing it. Notice that $G$ is $C$-layered with respect to any circle $C$ that contains a negative battery. \label{f4}}
\end{figure} 

Theorem \ref{allbat} is illustrated in Figure \ref{f4}. As seen in the figure, the presence of a negative battery allows for the presence of many other negative batteries. Interestingly, the presence of a positive battery severely restricts the possible locations of other positive batteries.

As an easy application of Theorem \ref{battery1} we notice that if $[e,C,+]$ is a positive battery, every edge contained in either handle of $C$ is a positive battery, and no other edge of $C$ is a positive battery. We will now study edges that are contained in $\Sigma{\setminus}C$.

\begin{lem}\label{balbridge} If $[e,C,+]$ is a positive battery and $D$ is a bridge of $C$, then $D$ is balanced. 

\begin{proof} Suppose $A$ is a negative circle contained in $D$. Let $P_0$ be a path through $D$ that intersects $A$ in a path $A_0$ with at least one edge. Then, define $A_1$ so that $A = A_0 \cup A_1$. Since $A$ is negative, $A_0$ and $A_1$ have opposite sign. Define $P_1$ to be the path obtained by replacing $A_0$ with $A_1$ in $P_0$. Then, $P_0$ and $P_1$ have opposite sign---impossible by Theorem \ref{battery1}. 

\end{proof}

\end{lem}

\begin{thm}\label{3bridge} Suppose $[e,C,+]$ is a positive battery. If $C$ has 3 or more bridges, no edge outside of $C$ is a positive battery.

\begin{proof} Let $f$ be an edge of $\Sigma$ not contained in $C$. Since $[e,C,+]$ is a positive battery, $\Sigma$ is $C$-layered and hence all bridges of $C$ have two vertices of attachment and are pairwise non-crossing. Let $D_1, D_2$, and $D_3$ be three bridges of $C$ with labeling chosen so that $f \in D_1$. For convenience, we switch $\Sigma$ so that $C$ is all-positive. By Theorem \ref{battery1}, performing this switch makes all paths through $D_1, D_2$, and $D_3$ negative. Let $P_1$ be such a path through $D_1$ that contains $f$, and let $P_2$ and $P_3$ be paths through $D_2$ and $D_3$ respectively. Then, $P_1$ and $P_2$, together with two (possibly trivial) paths through $C$ between the vertices of attachment of $D_1$ and $D_2$, form a positive circle containing $f$. Similarly, $P_1$ and $P_3$ are contained in a positive circle containing $f$. Thus, we have found two positive circles containing $f$---it is not a positive battery.
\end{proof}
\end{thm}

If $C$ has less than $3$ bridges  it is possible to have positive batteries outside of $C$, though specific types of bridges are required.

\begin{thm}\label{2bridge} Suppose $[e,C,+]$ is a positive battery and $C$ has exactly $2$ bridges. If both bridges are paths, every edge contained in either bridge is a positive battery. Otherwise, no edge in either bridge is a positive battery.

\begin{proof} Let $D_1$ and $D_2$ be the bridges of $C$. Once again we switch so that $C$ is all-positive. If $D_1$ and $D_2$ are paths, then they are negative paths by Theorem \ref{battery1}. In this case, $D_1$ and $D_2$ along with the (possibly trivial) paths in $C$ between their vertices of attachment form a positive circle. The fact that this is the only positive circle containing $D_1$ and $D_2$ is a matter of inspection.

Now, suppose that one of the bridges, say $D_1$, is not a path. Since we are working within a block, $D_1$ must contain a circle $A$. By Lemma \ref{balbridge}, $A$ is a positive circle. As in the proof of Lemma \ref{balbridge}, let $P_0$ be a path through $D_1$ that intersects $A$ in a path $A_0$ that contains at least one edge. Define $A_1$ so that $A = A_0 \cup A_1$. Let $P_1$ be the path obtained by replacing $A_0$ with $A_1$ in $P_0$. Let $P_2$ be a path through $D_2$. Then, $P_0, P_1$, and $P_2$ along with the appropriate paths through $C$ form a theta graph with three positive circles, forcing every edge in $P_0, P_1$ and $P_2$ to be in at least $2$ positive circles. Clearly for any edge $f$ outside of $C$, a choice of $P_0,P_1$, and $P_2$ can be found containing $f$.
\end{proof}
\end{thm}

Finally, we investigate the case where $C$ has only one bridge.

\begin{thm} Suppose $[e,C,+]$ is a positive battery and $C$ has exactly $1$ bridge $D$. An edge $f$ of $D$ is a positive battery if and only if it is contained in a single circle in $D$.

\begin{proof} Since $D$ is balanced (Lemma \ref{balbridge}), any positive battery contained in $D$ is contained in a single circle in $D$. In the other direction, suppose $f$ is contained in a single (positive) circle in $D$. Then, any other circle containing $f$ must consist of a path through $D$ and a path through $C$. This will form a negative circle by Theorem \ref{battery1}. 
\end{proof}
\end{thm}

\section{Vertex Batteries}\label{sec3}  Now let us examine a similar question to the one discussed in the previous section---when does a vertex of $\Sigma$ lie in a unique signed circle $C$? We will call such a vertex a \emph{vertex battery}, written $[v,C,-]$ or $[v,C,+]$ for the negative and positive varieties, respectively. We cannot assume that $G$ is a block, since a vertex, unlike an edge, may be contained in many blocks at once (if the vertex is a cutpoint). However, an acyclic block is of no interest to us. Thus, for the rest of this section we assume that every block containing $v$ contains a circle. 

First, we will study the problem of determining if $v$ is a negative vertex battery. It is clear that if $v$ is a negative vertex battery, it is contained in one and only one unbalanced block $B$. Moreover, there is a restriction on $\deg_B(v)$, the degree of $v$ in $B$. 

\begin{lem}\label{vertdeg} Let $B$ be an unbalanced block of $\Sigma$. If $\deg_B(v) \geq 3$, then $v$ is not a negative vertex battery.
\begin{proof} Certainly $v$ is contained in some negative circle $C$ in $B$. Thus, consider the two edges incident to $v$, say $e_1$ and $e_2$, that are contained in $C$. There is a third edge $e_3 \in B$ incident with $v$, and therefore there must be some theta graph $\Theta \subseteq B$ such that $C \subset \Theta$ and also $e_3 \in \Theta$. Therefore, $v$ is contained in all three circles of $\Theta$. Since $C$ is negative, $\Theta$ contains an additional negative circle which contains $v$. 
\end{proof}
\end{lem}

Assuming $v$ is a negative vertex battery, let $B$ be the single unbalanced block containing $v$, and let $C$ be the negative circle of $B$ containing $v$. By Lemma \ref{vertdeg}, $\deg_B(v)=2$, so let $v_1$ and $v_2$ be the two neighbors of $v$ in $B$. We can \emph{suppress} $v$---that is, replace the edges $v_1v$ and $v_2v$ with the single edge $e_v{:}v_1v_2$, setting $\sigma(e_v)=\sigma(v_1v)\sigma(v_2v)$. We write the result of the suppression as $B {\downarrow} v$. Thus, there is a bijective correspondence between the signed circles of $B$ and the signed circles of $B {\downarrow} v$. Therefore, we have the following theorem, which reduces the vertex problem to the edge problem.
\begin{thm}\label{uniquecharvtx} Let $\Sigma$ be a signed graph, and let $v$ be a vertex of $\Sigma$. Then $[v,C,-]$ is a negative vertex battery if and only if
\begin{enumerate}
\item The block $B$ of $\Sigma$ containing $v$ and $C$ is the only unbalanced block containing $v$, and $\deg_B(v)=2$.
\item $[e_v,C{\downarrow}v,-]$ is a negative edge battery in $B {\downarrow} v$. 
\end{enumerate}
\end{thm}

We will now turn our attention to the problem of determining if a given vertex is a positive vertex battery. Our characterization of positive vertex batteries is analagous to our characterization of negative vertex batteries, though slightly more complicated. Just as a negative vertex battery may be contained in many blocks in which it is contained in only positive circles (balanced), but only one block in which it is contained a negative circle, a positive vertex battery may be contained in many blocks in which it is contained in only negative circles, but only one block in which it is contained in a positive circle. 

We will first describe when an edge is contained in only negative circles and then convert this to a description of vertices. 
\begin{lem}\label{negativebalance} An edge of $\Sigma$ is contained in only negative circles if and only if it is a balancing edge in the block containing it.

\begin{proof} Let $e$ be an edge of $\Sigma$ and let $B$ be the block contaning it. First, if $e$ is a balancing edge in $B$, then $e$ is clearly contained in only negative circles (switch so that $e$ is the only negative edge). 

In the other direction, suppose that $e$ is contained only in negative circles. Consider $B{\setminus}e$. Suppose there is a negative circle $C$ contained in $B{\setminus}e$. Then, since $B$ is a block, there is a theta graph containing both $C$ and $e$, forcing $e$ to be in a positive circle---impossible. Thus, $B{\setminus}e$ is balanced and can hence be switched to all-positive by a switching function $\zeta$. However, applying $\zeta$ to $B$ instead will leave $e$ as a negative edge, since $e$ is contained in only negative circles. Thus, $e$ is a balancing edge in $B$. 
\end{proof}
\end{lem}

We can now characterize the type of block in which $v$ is contained only in negative circles.

\begin{lem}\label{vertonlyneg} A vertex $v$ is contained only in negative circles in $B$ if and only if $v$ is degree $2$ in $B$, and $e_v$ is a balancing edge in $B{\downarrow}v$.

\begin{proof} If $v$ is degree $2$ in $B$ and $e_v$ is a balancing edge, then clearly $v$ is contained in only negative circles by Lemma \ref{negativebalance}.

Conversely, if $v$ is contained in only negative circles, then $B$ is not balanced. If $\deg_B(v)\geq 3$, then $v$ is contained in a negative circle $C$ and furthermore $v$ is a degree $3$ vertex in a theta graph containing $C$. Hence, $v$ is contained in a positive circle---impossible. Thus, $\deg_B(v)=2$, and we can suppress $v$ to obtain the result.

\end{proof}
\end{lem}

Now we will describe the blocks in which $v$ is contained in exactly one positive circle. This along with Lemma \ref{vertonlyneg} will complete the description of positive vertex batteries. 

\begin{lem}\label{vertonepos} A vertex $v$ is contained in a unique positive circle $C$ in a block $B$ if and only if either:
\begin{enumerate}
\item $B=C$ is a balanced circle.
\item $B$ is unbalanced, $\deg_B(v)=2$, $B{\downarrow}v$ is $C$-layered, and $e_v$ is a positive edge battery in $B{\downarrow}v$.
\item $B$ is unbalanced, $\deg_B(v)=3$, $B$ is $C$-layered with exactly one bridge, $v$ is a vertex of attachment of the bridge, and all edges of $C$ are positive edge batteries.

\end{enumerate}

\begin{proof} If $B$ and $v$ satisfy either of the first two conditions above, then clearly $v$ is contained in a unique positive circle in $B$. If $B$ and $v$ satisfy the third condition, observe that any circle besides $C$ that contains $v$ must consist of a path through the single bridge and a handle of $C$---a negative circle. 

In the other direction, suppose $v$ is contained in a unique positive circle in $B$. If $B$ is balanced, it consists of only a single positive circle or else $v$ is contained in multiple positive circles. 

So, suppose that $B$ is unbalanced and consider the case where $\deg_B(v) \geq 4$. Let $e_1,...,e_4$ be four edges incident with $v$, so that $e_1$ and $e_2$ are contained in $C$. Since $B$ is a block, we can find a theta graph $\Theta$ that contains $C$, and $e_3$ (up to choice of labels). Furthermore, there is a path $P$ containing $e_4$ starting at $v$ and  ending at a vertex of $\Theta$ (besides $v$). It is easy to see that $\Theta \cup P$ contains at least two positive circles that contain $v$. 

Now we consider the case where $B$ is unbalanced and $\deg_B(v)=3$. Find a theta graph $\Theta \subseteq B$ that contains $C$ such that $\deg_{\Theta}(v)=3$. Observe that $\Theta$ is $C$-layered, with a single bridge $P$ consisting of a path. There can be no path connecting an internal vertex of $P$ to an internal vertex of a handle of $C$, or else we create a subdivision of $K_4$, forcing $v$ to be in more than one positive circle. There can be no path besides $P$ connecting two vertices of $C$, or else there is a second positive circle containing $v$. Thus, any path besides $P$ connecting two vertices of $\Theta$ must have both its endpoints in $P$ and be otherwise disjoint from $\Theta$. Since $\deg_{\Theta}(v)=3$, no such path will contain $v$. The union of all such paths forms a single bridge of $C$, of which $v$ is a vertex of attachment.
\end{proof}
\end{lem}

Finally, we combine the results of this section to give a complete description of whether or not $v$ is a positive vertex battery. As a reminder, we assume that each block containing $v$ contains at least one circle.

\begin{thm} Let $\Sigma$ be a signed graph, and let $v$ be a vertex of $\Sigma$. Then $[v,C,+]$ is a positive vertex battery if and only if:

\begin{enumerate}
\item The block containing $v$ and $C$ is the unique block of the form described in Lemma \ref{vertonepos}.
\item In any other block $B$ containing $v$, $v$ is degree $2$ and $e_v$ is a balancing edge in $B{\downarrow}v$.
\end{enumerate}

\begin{proof} A positive vertex battery must be contained in exactly one block in which it is contained in a unique positive circle. It must be contained in no positive circles in all other blocks containing it. Thus, the proof is a combination of Lemma \ref{vertonepos} and Lemma \ref{vertonlyneg}.

\end{proof}

\end{thm}

\end{document}